\documentclass[12pt]{amsart}
\usepackage[all]{xy}           
\usepackage{amssymb}           
\usepackage{hyperref}
\usepackage{eucal}
\usepackage{graphicx}
\usepackage{epsfig}
\usepackage[usenames,dvipsnames]{color}
\usepackage{color}

\usepackage{comment}

\usepackage{epstopdf}

\newtheorem {theorem} {Theorem}
\newtheorem {proposition} [theorem]{Proposition}
\newtheorem {corollary} [theorem]{Corollary}
\newtheorem {lemma}  [theorem]{Lemma}

\newcommand{\R}{\mathbb{R}}      
\newcommand{\df}[2]{\displaystyle{\frac{#1}{#2}}}

\newcount\ancho \newcount\anchom \newcount\anchoa
\newcount\anchob \newcount\altura

\begin{document}

\title{Trapezoid central configurations}

\author[M. \ Corbera]{M.\ Corbera}
\address{M.\ Corbera:
Departamento de Matem\'aticas \\
Universidad de Vic, Spain} \email{montserrat.corbera@uvic.cat}

\author[J.M. \ Cors]{J.M. \ Cors}
\address{J.M.\ Cors:
Departamento de Matem\'aticas \\
Universidad Polit\'ecnica de Catalunya, Spain}
\email{cors@epsem.upc.edu}

\author[J. \ Llibre]{J. \ Llibre}
\address{J.\ Llibre:
Departamento de Matem\'aticas \\
Universidad Aut\'onoma de Barcelona, Spain}
\email{jllibre@mat.uab.cat}

\author[E. P\'erez-Chavela]{E. P\'erez-Chavela}
\address{E. P\'erez-Chavela: Departamento de Matem\'aticas, ITAM\\
R\'io Hondo 1, Col. Progreso Tizap\'an, Ciudad de  M\'exico 01080,
M\'exico} \email{ernesto.perez@itam.mx}



\begin{abstract}
We classify all planar four--body central configurations where two
pairs of the bodies are on parallel lines. Using Cartesian
coordinates, we show that the set of four--body trapezoid central
configurations with positive masses forms a two--dimensional surface
where two symmetric families, the rhombus and isosceles trapezoid,
are on its boundary. We also prove that, for a given position of the
bodies, in some cases an specific order of the masses determine the
geometry of the configuration, namely acute or obtuse trapezoid
central configuration. We also prove the existence on non--symmetric
trapezoid central configuration with two pairs of equal masses.
\end{abstract}

\maketitle

\section{Introduction}\label{sec:intro}

Central configurations are particular positions of the masses in the
Newtonian $n$--body problem, where the position and acceleration
vectors with respect to the center of masses are proportional, with
the same constant of proportionality for all masses. They play an
important role in celestial mechanics because, among other
properties, they generate the unique known explicit solutions in the
$n$--body problem for $n \geq 3$. For general information about
central configurations see for instance Albouy and Chenciner
\cite{AC}, Hagihara \cite{hag}, Moeckel \cite{rick2}, Saari
\cite{saa, saari}, Schmidt \cite{sch}, Smale \cite{Sm1, Sm2} and
Wintner \cite{wintner}.

More precisely we consider the planar $n$--body problem
$$
m_k\,\ddot{\mathbf{q}}_k=-\sum_{\scriptsize \begin{array}{c}
j=1\\j\ne k\end{array}}^n
G\,m_k\,m_j\,\df{\mathbf{q}_k-\mathbf{q}_j}{|\mathbf{q}_k-\mathbf{q}_j|^3}
\ ,
$$
$k=1,\dots,n$, being $\mathbf{q}_k\in\mathbb{R}^2$ the position
vector of the punctual mass $m_k$ in an inertial coordinate system,
and $G$ is the gravitational constant that we can take equal to one
by choosing conveniently the unit of time. The \emph{configuration
space} of the planar $n$--body problem is
$$
\mathcal{E}=\{(\mathbf{q}_1,\dots,\mathbf{q}_n)\in\mathbb{R}^{2 n}\:
:\:  \mathbf{q}_k\ne \mathbf{q}_j,\ \mbox{for } k\ne j\} .
$$

A configuration of the $n$ bodies
$(\mathbf{q}_1,\dots,\mathbf{q}_n)\in\mathcal{E}$ is \emph{central}
if there is a positive constant $\lambda$ such that
\begin{equation}
\ddot{\mathbf{q}}_k=-\lambda\,\left(\mathbf{q}_k-\mathbf{c}\right)\
, \label{sistema}
\end{equation}
for $k=1,\dots,n$, being $\mathbf{c}$ the position vector of the
center of mass of the system, which is defined by $
\mathbf{c}=\sum_{k=1}^n m_k \mathbf{q}_k/\sum_{k=1}^n m_k. $

Two planar central configurations are \emph{equivalent} if there is
a homothecy of $\mathbb{R}^2$ and a rotation of $SO(2)$ with respect
to the center of mass which send one into the other. Since this
relation is of equivalency, in what follows we shall consider the
classes of equivalency of central configurations.

The complete set of planar central configurations of the $n$--body
problem is only known for $n=2,3$. For $n=2$ there is only one class
of central configurations. For each choice of three positive masses
there are five classes of central configurations of the three--body
problem, the three collinear central configurations found in 1767 by
Euler \cite{eul}, and the two equilateral triangle central
configurations found in 1772 by Lagrange \cite{La}.

When $n>3$ there are many partial results for the number of classes
of central configurations of the $n$--body problem. In 1910 Moulton
\cite{Mo} showed that there exists exactly $n!/2$ classes of
collinear central configurations for any given set of positive
masses, one for each ordering of the masses on the straight line
modulo a rotation of $\pi$ radians. A lower bound of the number of
planar non--collinear central configurations was obtained by Palmore
in \cite{Palmore}.

Although the set of all planar central configurations of the
four--body problem is not completely known, we can find in the
literature several papers that provide the existence and
classification of central configurations of the four--body problem
in some particular cases. For instance, a complete numerical study
for the number of classes of central configurations for $n=4$ and
arbitrary masses was done by Sim\'o in \cite{Simo}. A computer
assisted proof of the finiteness of the number of central
configurations for $n=4$ and any choice of the masses was given by
Hampton and Moeckel \cite{HM}. Later on Albouy and Kaloshin
\cite{AK} proved this result analytically, and extend it for for $n=5$,
except for a zero measure set in the masses space .

Assuming that every central configuration of the four--body problem
with equal masses has an axis of symmetry Llibre in \cite{Llibre4}
obtained all the planar central configurations of the four--body
problem with equal masses by studying the intersection points of two
planar curves. Later on Albouy in \cite{A1,A2} gave a complete
analytic proof of this result.

When one of the four masses is sufficiently small Pedersen
\cite{Pedersen}, Gannaway \cite{Gann} and Arenstorf \cite{Arenstorf}
numerically and analytically obtained the number of its classes of
central configurations. These studies were completed later on by
Barros and Leandro in \cite{Barros1} and \cite{Barros2}.

A central configuration is called \emph{kite} if it has an axis of
symmetry passing through two non--adjacent bodies. The kite
non--collinear classes of central configurations having some
symmetry for the four--body problem with three equal masses where
characterized by Bernat et al. in \cite{BLP}, see also Leandro
\cite{Leandro}. The characterization of the convex central
configurations with an axis of symmetry and the concave central
configurations of the four--body problem when the masses satisfy
that $m_1=m_2\ne m_3=m_4$ was done by \'Alvarez and Llibre in
\cite{AL}.

A planar configuration of the four--body problem can be classified
as either \emph{convex} or \emph{concave}. A configuration  is
\emph{convex} if none of the bodies is located in the interior of
the triangle formed by the others. A configuration
 is \emph{concave} if one of the bodies is in the
interior of the triangle formed by the others.

In \cite{MB} MacMillan and Bartky shown that for any assigned order
of any four positive masses there is a convex planar central
configuration of the four--body problem with that order. Later on,
Xia \cite{xia} provided a simpler proof of this result. The
following {\it convex conjecture} stated by Albouy and Fu in
\cite{AF} (see also \cite{MB, PS}) is well known between the
community working in central configurations: {\it  For the planar
four--body problem there is a unique convex  central
configuration of the four--body problem for each ordering of the
masses in the boundary of its convex hull}.

Already, MacMillan and Bartky in \cite{MB} proved that there exists a
unique isosceles trapezoid central configuration of the four--body
when two pairs of equal masses are located at adjacent vertices.
Later on Xie in \cite{XieZhifu} reproved this result.

The following subconjecture of the convex conjecture is also well
known: {\it For the planar four--body problem there exists a unique convex
central configuration having two pairs of equal masses
located at the adjacent vertices of the configuration and it is an
isosceles trapezoid}.

In \cite{LongSun} Long and Sun shown that any convex central
configuration with masses $m_1=m_2<m_3=m_4$ located at the opposite
vertices of a quadrilateral and such that the diagonal corresponding
to the mass $m_1$ is not shorter than the one corresponding to the
mass $m_3$, has a symmetry and the quadrilateral is a rhombus. This
result was extended by P\'erez--Chavela \cite{PS} and Santoprete to
the case where two of the masses are equal and at most, only one of
the remaining mass is larger than the equal masses. Moreover, they
proved that there is only one convex central configuration when the
opposite masses are equal and it is a rhombus. Later on Albouy et.
al. in \cite{AFS} shown that in the four--body problem a convex
central configuration is symmetric with respect to one diagonal if
and only if the masses of the two particles on the other diagonal
are equal. If these two masses are unequal, then the less massive
one is closer to the former diagonal.

Using the results on the symmetries mentioned in the previous
paragraph Corbera and Llibre \cite{CoL} gave a complete description
of the families of central configurations of the four--body problem
with two pairs of equals masses and two equal masses sufficiently
small, proving for these masses the convex conjecture and the
subconjecture. More recently, the subconjecture was proved for
arbitrary masses by Fernandes et al. in \cite{FLM}.

The co-circular classes of central configurations of the four--body
problem, i.e. when the four masses are on a circle have been studied
by Cors and Roberts in \cite{cors}.

A {\it trapezoid} is a convex quadrilateral with at least one pair
of parallel sides. The parallel sides are called the {\it bases} of
the trapezoid and the other two sides are called the {\it lateral
sides}. See Figure~\ref{figtrap} for the classification of the nine
trapezoids, namely:
\begin{itemize}
\item An {\it acute trapezoid} has two adjacent acute angles on its
longer base edge.

\item An {\it obtuse trapezoid} has one acute and one obtuse angle
on each base.

\item A {\it right trapezoid} has two adjacent right angles.

\item An {\it isosceles trapezoid} is an acute trapezoid if its
lateral sides have the same length, and the base angles have the
same measure.

\item A {\it $3$--sides equal trapezoid} is an isosceles trapezoid
with three sides of the same length.

\item A {\it parallelogram} is an obtuse trapezoid with two pairs
of parallel sides.

\item A {\it rhombus} is a parallelogram with the four sides with
the same length.

\item A {\it rectangle} is a parallelogram with four right angles.

\item A {\it square} is a rectangle with the four sides with
the same length.
\end{itemize}

\begin{figure}
\unitlength=1cm
\begin{center}
\begin{picture}(12,7)(0,-0.7)
\put(0,0){\includegraphics[width=12cm]{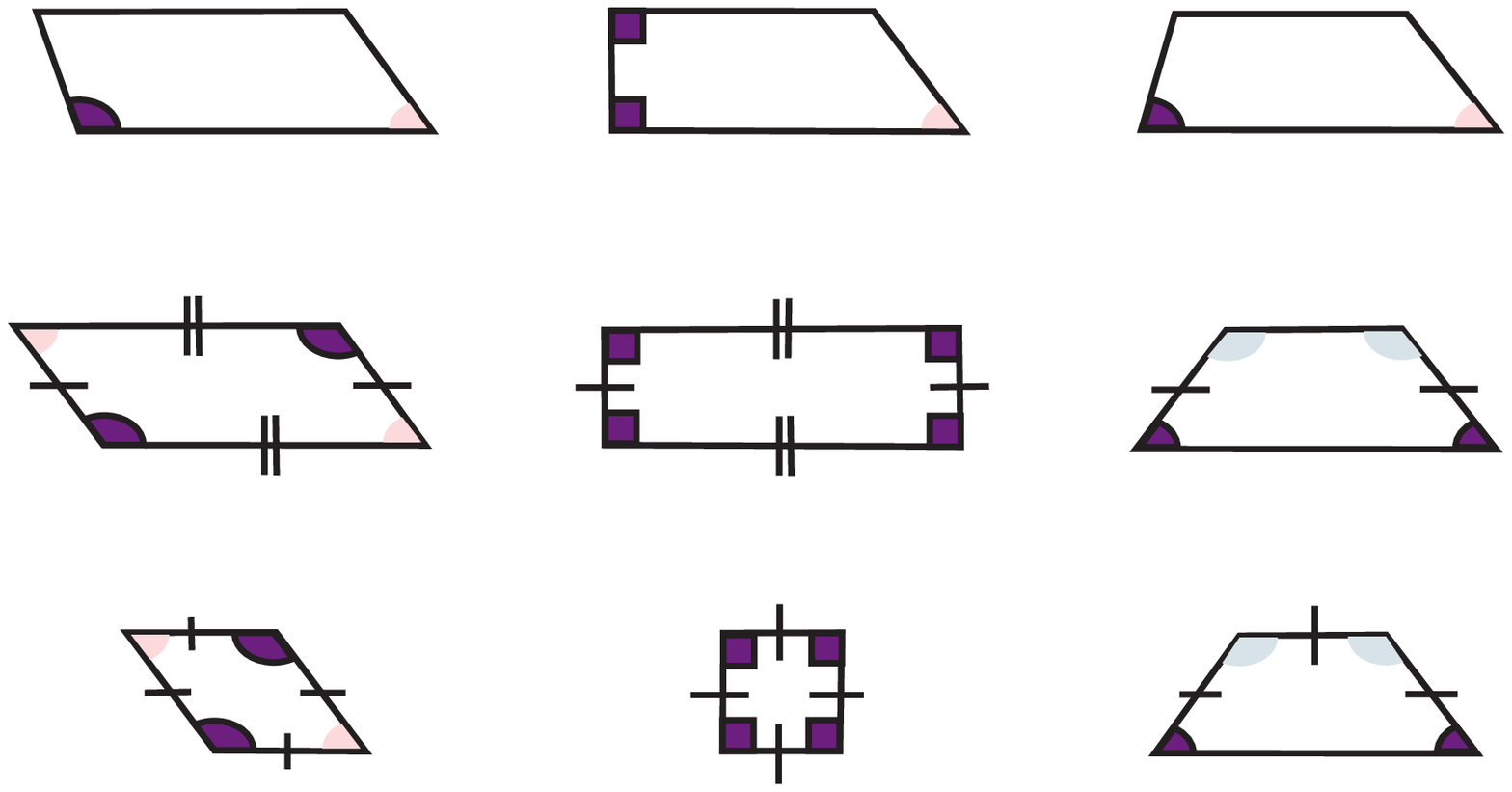}}
\put(0.7,4.5){Obtuse trapezoid} \put(0.9,2){Parallelogram}
\put(1.5,-0.4){Rhombus} \put(5.2,4.5){Right trapezoid}
\put(5.5,2){Rectangle} \put(5.7,-0.4){Square} \put(9.3,4.5){Acute
trapezoid} \put(9.2,2){Isosceles trapezoid} \put(9.1,-0.4){Isosceles
trapezoid} \put(9.1,-0.8){with 3 equal sides}
\end{picture}
\end{center}
\caption{Classification of Trapezoids} \label{figtrap}
\end{figure}

In this paper we are interested in studying trapezoid central
configurations. See \cite{santoprete} for a really fresh work in the
same topic. In section \ref{equations} we derive the equations for
the trapezoid central configurations in terms of the mutual
distances. In section \ref{no} we prove that not all the trapezoid
configurations are realizable. In section \ref{si} we characterize,
using cartesian coordinates, the set of positions that yield to
trapezoid central configurations with positive masses. In section
\ref{right} we prove that there exist a one--parameter family of
right trapezoid central configurations. Finally, in section
\ref{masses} we study the set of positive masses which yields to a
trapezoid central configurations. We prove, in contrast to the
co--circular case, that two pair of equal masses do not imply that
the central configuration has some symmetry.

\section{Preliminaries}\label{prelim}

The central configurations (in what follows simply CC by short) can
be described in terms of Lagrange multipliers. We denote by ${\bf
q}=({\bf q}_1, {\bf q}_2, {\bf q}_3, {\bf q}_4) \in (\R^2)^4$ the
position of four positive masses $m_1, m_2, m_3, m_4$ on the plane
and by $r_{ij}=||\mathbf{q}_i-\mathbf{q}_j||$ the mutual distances
between the $i$--th and the $j$--th bodies. The vector ${\bf q}$ is a CC
of the $4$--body problem if it satisfies the following algebraic
equation for some value of $\lambda$ (the Lagrange multiplier)
\begin{equation}\label{eq:cc}
\nabla U ({\bf q}) + \lambda \, \nabla I ({\bf q}) = 0,
\end{equation}
where $U$ is the Newtonian potential
\begin{equation}\label{eq:pot}
U({\bf q}) = \sum_{i < j}^n  \; \frac{m_i m_j}{r_{ij}},
\end{equation}
$I({\bf q})$ is the moment of inertia, which represents the size of
the system,
\begin{equation}\label{eq:iner}
I({\bf q}) =\frac{1}{2}\sum_{i=1}^n
m_i||\mathbf{q}_i-\mathbf{c}||^2= \frac{1}{2M} \sum_{1\le i < j\le
n} \; m_i m_j r_{ij}^2,
\end{equation}
${\bf c}$ is the center of mass of the system, and $M = m_1 + \cdots
+ m_n$ is the total mass (see \cite{meyer} for more details).

We observe that generically six mutual distances describe a
tetrahedron in $\R^3$, since in this work we are  interested in
planar CC, when we write equation (\ref{eq:cc}) in terms of mutual
distances, we must add a constraint to maintain the particles on a
plane. This constraint arises setting the volume of the tetrahedron
equals to zero. Denoting as ${\bf r} = (r_{12}, r_{13}, r_{14},
r_{23}, r_{24}, r_{34}) \in \mathbb{R}^{+^6}$ the vector of mutual
distances,  it is well know in the literature (see for instance
\cite{saari, sch}), that the volume $V$ of a tetrahedron is given by
the Cayley-Menger determinant
$$
V^2 \; = \;  \frac{1}{288}\left|  \begin{array}{ccccc}
        0 & 1  & 1   & 1  & 1 \\
             1 & 0  & r_{12}^2 & r_{13}^2 & r_{14}^2  \\
                  1 & r_{12}^2 &  0  & r_{23}^2  & r_{24}^2 \\
                  1 & r_{13}^2 & r_{23}^2 & 0 & r_{34}^2    \\
                  1 & r_{14}^2 & r_{24}^2 & r_{34}^2 & 0
              \end{array} \right|.
$$

{From} now on we will assume that $V = 0$. Also in order to avoid
collinear configurations we impose that all triples of mutual
distances satisfy  strictly the triangle inequality (see \cite{cors}
for more details).

Let $A_i$ be the oriented area of the triangle formed by the
configuration ${\bf q}$ where the point ${\bf q}_i$ is deleted,  and
let $\Delta_i=(-1)^{i+1}A_i$. Since $A_i>0$ when the vertices are
ordered sequentially counterclockwise, for a convex quadrilateral
ordered sequentially counterclockwise we obtain $\Delta_1, \Delta_3
>0$ and $\Delta_2, \Delta_4 < 0,$ satisfying the equation
$$\Delta_1 + \Delta_2 + \Delta_3 + \Delta_4 = 0.$$
In 1900 Dziobeck \cite{dzio} proved for planar CC that
\begin{equation*}\label{eq:dzi}
\frac{ \partial V}{ \partial r_{ij}^2} \; = \;  - 32 \, \Delta_i
\Delta_j \, .
\end{equation*}
{From} this equality we obtain
$$\frac{ \partial V}{ \partial r_{ij}}  =
\frac{ \partial V}{ \partial r_{ij}^2} \cdot \frac{ d (r_{ij}^2)}{d
r_{ij}} =    - 64 r_{ij} \,  \Delta_i  \Delta_j .$$

Fixing the moment of inertia $I = I_0$ and applying Lagrange
multipliers,  we have that the planar non-collinear CC are the
critical points of the function
\begin{equation}\label{eq:cc-2}
U + \lambda M (I - I_0) - \sigma V.
\end{equation}

Taking the partial derivatives and using the six mutual distances as
variables we obtain
\begin{eqnarray}\label{eq:part-der}
m_1 m_2 (r_{12}^{-3} - \lambda) = 64\sigma \Delta_1\Delta_2, & &
m_3 m_4 (r_{34}^{-3} - \lambda)
= 64\sigma  \Delta_3\Delta_4 ,  \nonumber \\
m_1 m_3 (r_{13}^{-3} - \lambda) = 64\sigma \Delta_1\Delta_3, &  &
m_2 m_4 (r_{24}^{-3} - \lambda)
= 64\sigma \Delta_2\Delta_4 ,   \\
m_1 m_4 (r_{14}^{-3} - \lambda) = 64\sigma \Delta_1\Delta_4, & &
m_2 m_3 (r_{23}^{-3} - \lambda) = 64\sigma \Delta_2\Delta_3 .
\nonumber
\end{eqnarray}

Grouping the above equations by row, so that the product of the
right-hand side is simply $(64\sigma)^2 \Delta_1\Delta_2
\Delta_3\Delta_4$, and since the masses are positive we obtain the
well known Dziobeck relation
\begin{equation}\label{eq:dzio-rel}
(r_{12}^{-3} - \lambda)(r_{34}^{-3} - \lambda) = (r_{13}^{-3} -
\lambda)(r_{24}^{-3} - \lambda) =  (r_{14}^{-3} -
\lambda)(r_{23}^{-3} - \lambda),
\end{equation}
which must be satisfied for any planar 4-body CC. Solving each of
the three pairs of equations  for $\lambda$ we obtain
\begin{equation}\label{eq:lambda}
\begin{array}{rl}
\lambda \; =& \dfrac{  r_{12}^{-3} \, r_{34}^{-3}  -  r_{13}^{-3} \,
r_{24}^{-3} }{r_{12}^{-3} + r_{34}^{-3} - r_{13}^{-3} - r_{24}^{-3}}\\
= & \dfrac{  r_{13}^{-3} \, r_{24}^{-3}  - r_{14}^{-3} \,
r_{23}^{-3} }{r_{13}^{-3} + r_{24}^{-3} - r_{14}^{-3} - r_{23}^{-3}}\\
= & \dfrac{  r_{14}^{-3} \, r_{23}^{-3}  - r_{12}^{-3} \,
r_{34}^{-3} }{r_{14}^{-3} + r_{23}^{-3} - r_{12}^{-3} -
r_{34}^{-3}}.
\end{array}
\end{equation}

If we set
$$
\begin{array}{ccc}
s_1 = r_{12}^{-3} + r_{34}^{-3}, & \qquad  &  p_1 = r_{12}^{-3}
r_{34}^{-3},
\\[0.07in]
s_2 = r_{13}^{-3} + r_{24}^{-3}, &   &  p_2 = r_{13}^{-3}
r_{24}^{-3},
\\[0.07in]
s_3 = r_{14}^{-3} + r_{23}^{-3}, &   &  p_3 = r_{14}^{-3}
r_{23}^{-3},
\end{array}
$$
then equation (\ref{eq:lambda}) can be written as
\begin{equation}\label{eq:lambda-bis}
\lambda \; = \;  \frac{p_1 - p_2}{s_1 - s_2}  \; = \;  \frac{p_2 -
p_3}{s_2 - s_3}  \; = \;  \frac{p_3 - p_1}{s_3 - s_1},
\end{equation}
which means that $(s_1,p_1), (s_2,p_2), (s_3,p_3)$, viewed as points
in $\mathbb{R}^{+2}$, must lie on the same line with slope
$\lambda$.  This in turn, is equivalent to
$$
\left|  \begin{array}{ccc}
         1      & 1       &  1   \\[0.05in]
                 s_1    & s_2     & s_3  \\[0.05in]
             p_1    & p_2     & p_3
              \end{array} \right|
              \; = \; 0,
$$
a representation that allows to write Dziobeck equation
(\ref{eq:dzio-rel}) as the nice factorization
$$
D=(r_{13}^3 - r_{12}^3)(r_{23}^3 - r_{34}^3) (r_{24}^3 - r_{14}^3) -
(r_{12}^3 - r_{14}^3) (r_{24}^3 - r_{34}^3) (r_{13}^3 - r_{23}^3)=0
$$

The Dziobeck equation $D=0$ must be satisfied for the six mutual
distances of every four-body planar central configuration.

\section{Equations of trapezoidal central configurations}\label{equations}

We consider four positive masses $m_1,m_2,m_3,m_4$ located at the
vertices of a trapezoid, i.e. located  by pairs on two parallel
lines, which without loss of generality we assume are vertical. Since
the central configurations are invariant under homotheties we can
take the distance between the two parallel lines equal to one, after
normalizing the unity of mass we can assume that $m_1=1$ located at
the bottom part of the left line,  $m_2$ above $m_1$ on the same
line and $m_3$ above $m_4$ on the right line (see Fig.~\ref{fig1}).
From now on we will use this ordering and normalization of the units
of mass and length in this work.

\begin{figure}
\unitlength=1cm
\begin{picture}(4,5)(0,-1)
\thicklines \put(0.5,0){\line(0,1){4}} \put(2.5,0){\line(0,1){4}}
\put(0.5,0.5){\circle*{0.18}} \put(0.5,3.5){\circle*{0.18}}
\put(2.5,2.5){\circle*{0.18}} \put(2.5,1){\circle*{0.18}}
\put(-1,0.4){$m_1=1$} \put(-0.3,3.4){$m_2$} \put(2.8,1.0){$m_4$}
\put(2.8,2.4){$m_3$} \thinlines \put(0.5,0.5){\line(4,1){2}}
\put(0.5,3.5){\line(2,-1){2}}
\end{picture}
\caption{Central configurations with two parallel lines}
\label{fig1}
\end{figure}
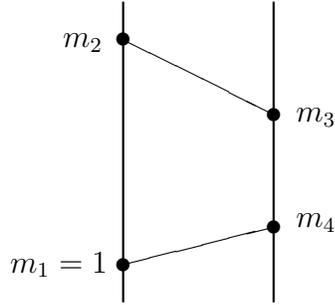

All trapezoidal central configurations are convex, so from the
results of McMillan \cite{MB}, we know first that the diagonals of
the respective quadrilateral are longer that any of the four sides,
that is
\begin{equation} \label{ineq}
r_{13}, r_{24}  \; >  \; r_{12}, r_{14}, r_{23}, r_{34};
\end{equation}
and second that the bigger and the smaller sides of the
quadrilateral correspond to opposite sides. We note that in the
restricted problem, i.e. when one or more masses are equal to zero,
one of the sides of the quadrilateral could be equal to one
diagonal.

\begin{lemma}
The biggest side of the quadrilateral is on the parallel lines.
\label{lemadesigualtat}
\end{lemma}

\begin{proof}
Assume that $r_{23}$ is the biggest side and that we exclude the
case where all the sides are equal, that is, the square. So, its
opposite side, $r_{14}$, has to be the smaller one, and
$$
r_{13}, r_{24}  \; >  \; r_{23} \; \ge  \; r_{12}, r_{34} \; \ge  \;
r_{14}>1.
$$
Then depending on the relative position of the four masses we have
the following four scenarios:
\begin{itemize}
\item[(a)] $r_{34}=r_{12}+\sqrt{ r_{23}^2-1}+\sqrt{ r_{14}^2-1},$
\item[(b)] $r_{34}=r_{12}+\sqrt{ r_{23}^2-1}-\sqrt{ r_{14}^2-1},$
\item[(c)] $r_{34}=r_{12}-\sqrt{ r_{23}^2-1}+\sqrt{ r_{14}^2-1},$
\item[(d)] $r_{34}=r_{12}-\sqrt{ r_{23}^2-1}-\sqrt{ r_{14}^2-1}.$
\end{itemize}
Notice that the cases where $m_3$ and $m_4$ are either both above
$m_2$ or both below $m_1$ are not possible because in these cases
one of the diagonals would be smaller than one of the sides.

In all scenarios we will arrive to a contradiction with the fact that
$r_{23}$ is the biggest side or $r_{14}$ the smaller one.

In the scenario (a), $r_{34}>r_{12}+\sqrt{ r_{23}^2-1}>1+\sqrt{
r_{23}^2-1}>r_{23}$.

For the scenarios (b) and (c) we shall use the following result: Let
$x,y\geq 1$ be two real numbers,
$$
x-\sqrt{ x^2-1}+\sqrt{ y^2-1}=y \quad\text{if and only if}\quad x=y.
$$
Moreover if $y>x$ then $x-\sqrt{ x^2-1}+\sqrt{ y^2-1}>y$, and if
$y<x$ then $x-\sqrt{ x^2-1}+\sqrt{ y^2-1}<y$.

In (b) $r_{34}>\sqrt{ r_{23}^2-1}+r_{14}-\sqrt{ r_{14}^2-1}>r_{23}$,
and in (c) $r_{34}<r_{23}-\sqrt{ r_{23}^2-1}+\sqrt{
r_{14}^2-1}<r_{14}$.

Finally, in (d) $r_{34}<r_{23}-\sqrt{ r_{23}^2-1}+\sqrt{
r_{14}^2-1}<r_{14}$ from scenario (c).

Similar argument works if $r_{14}>1$ is considered the biggest side.
\end{proof}

Without loss of generality we label the bodies so that $r_{12}$ is
the longest side. We can also assume that $r_{23}\geq r_{14}$ by an
appropriate relabeling. Indeed, equations (\ref{eq:part-der}) are
invariant if we interchange bodies $m_1$ and $m_2$ and bodies $m_3$
and $m_4$. The choice $r_{23}\geq r_{14}$, together with the fact
that $r_{12}$ is the longest side, implies the relation $r_{24}\geq
r_{13}$ between the two diagonals.

Summarizing, we have proved the following result.

\begin{lemma}\label{lemomega}
Labelling conveniently the bodies, the mutual distances that can
provide trapezoid central configurations can be restricted to the
following set
$$
\widetilde{\Omega}=\{{\bold r}\in\mathbb{{R}^{+}}^{6}:r_{24}\geq
r_{13}>r_{12}\geq r_{23}\geq r_{14}\geq r_{34}\}.
$$
\end{lemma}

Next we give the expression of the masses ratios for the trapezoid
central configurations on $\widetilde{\Omega}$. Taking into account
the sign of the areas $A_i$, we have $\Delta_1 = -r_{34}/2, \,
\Delta_2 = r_{34}/2, \, \Delta_3 = -r_{12}/2, \, \Delta_4 =
r_{12}/2$  (note that we have considered the bodies ordered
clockwise). Now from (\ref{eq:part-der}) we obtain the following
ratios of the masses
\begin{eqnarray}
\frac{m_2}{m_1} & = & \frac{  (\lambda - r_{13}^{-3})  }{
(r_{23}^{-3} - \lambda  ) } \; = \;
\frac{  (\lambda - r_{14}^{-3} ) }{ ( r_{24}^{-3} - \lambda ) },
\label{eq:m1m2} \\[0.07in]
\frac{m_3}{m_1} & = & \frac{  (r_{12}^{-3} - \lambda) \, r_{12} }{
(r_{23}^{-3} - \lambda) \,  r_{34}} \; = \;
\frac{  (r_{14}^{-3} - \lambda) \, r_{12} }{ (r_{34}^{-3} - \lambda)
\,  r_{34}},  \label{eq:m1m3}  \\[0.07in]
\frac{m_4}{m_1} & = & \frac{  (r_{12}^{-3} - \lambda) \, r_{12} }{
(\lambda - r_{24}^{-3}) \,  r_{34}}  \; = \; \frac{  (\lambda -
r_{13}^{-3}) \, r_{12} }{ (r_{34}^{-3} - \lambda) \,  r_{34}} .
\label{eq:m1m4}
\end{eqnarray}

We observe that the fact that all masses must be positive places
additional constraints on the mutual distances. Using  $\lambda =
(p_2 - p_3)/(s_2 - s_3)$ and $m_1=1$ into the first equation
in~(\ref{eq:m1m2}) and after some simplifications we obtain
\begin{equation}\label{eq:fin12}
m_2  =  \frac{  r_{23}^3 r_{24}^3 \, (r_{13}^3 - r_{14}^3)}{
r_{13}^3 r_{14}^3 \, (r_{24}^3 - r_{23}^3)}.
\end{equation}
Doing similar substitutions in~(\ref{eq:m1m3}) and~(\ref{eq:m1m4})
we get
\begin{equation}\label{eq:fin13}
m_3  =  \frac{  r_{23}^3 r_{34}^2 \, (r_{12}^3 - r_{14}^3)}{
r_{12}^2 r_{14}^3 \, (r_{23}^3 - r_{34}^3)},
\end{equation}
and
\begin{equation}\label{eq:fin14}
m_4  =  \frac{  r_{24}^3 r_{34}^2 \, (r_{13}^3 - r_{12}^3)}{
r_{12}^2 r_{13}^3 \, (r_{24}^3 - r_{34}^3)},
\end{equation}
respectively.

The masses of equations (\ref{eq:fin12}),  (\ref{eq:fin13}) and
(\ref{eq:fin14}) are positive and well-defined on
$\widetilde{\Omega}$, except when $r_{12}=r_{14}$ and
$r_{23}=r_{34}$ simultaneously. In that case, we use $\lambda = (p_1
- p_2)/(s_1 - s_2)$ into equation (\ref{eq:m1m3}) getting
\begin{equation}\label{eq:squa}
m_3 \;  = \; \frac{r_{34}^5(r_{14}^3 - r_{24}^3)( r_{14}^3
-r_{13}^3)}{r_{14}^5(r_{13}^3 - r_{34}^3)( r_{24}^3 -r_{34}^3)} ,
\end{equation}
which also is positive and well-defined on $\widetilde{\Omega}$.

In summary we have proved the next result.

\begin{lemma}\label{lemomega2}
Let
$$
\widetilde{\Omega}'=\{{\bold r}\in\mathbb{{R}^{+}}^{6}\ : {\bold
r}\in\widetilde{\Omega}\,\, and \,\, D=0\}.
$$
Any point in $\widetilde{\Omega}'$ defines a four--body trapezoid
central configuration with positive masses. Moreover, up to
relabelling and rescaling the set $\widetilde{\Omega}'$ contains all
trapezoid central configurations.
\end{lemma}

\section{The trapezoids which are not realizable as central
configuration}\label{no}

In this section we prove that the vertices of the parallelogram, the
rectangle and the $3$--sides equal trapezoid are not realizable as
central configurations of the four-body problem with the exception
of the square and the rhombus.

Assume that the bodies are ordered sequentially as in
Figure~\ref{fig1}.

\begin{proposition}\label{prop1nn}
In the planar four--body problem there are no parallelogram shape
central configurations with positive masses at their vertices,
excluding squares and rhombus.
\end{proposition}

\begin{proof}
In a parallelogram configuration $r_{12}=r_{34}$ and
$r_{23}=r_{14}$. From Lemma~\ref{lemomega2}, this parallelogram
could be realizable as a central configuration if
$r_{12}=r_{23}=r_{14}=r_{34}$, that is, if it is a rhombus or a
square.
\end{proof}

The next result is an immediate consequence of Proposition
\ref{prop1nn}.

\begin{corollary}
In the planar four--body problem there are no rectangle shape
central configurations with positive masses at their vertices.
\end{corollary}

\begin{proposition}\label{prop1n}
In the planar four--body problem there are no $3$--sides equal
trapezoid shape central configurations with positive masses at their
vertices, excluding the square.
\end{proposition}

\begin{proof}
A $3$--sides equal trapezoid is in particular an isosceles
trapezoid, so the length of its diagonals are equal. Assume that
$r_{12}$ is the longest exterior side, that the equal sides are
$r_{23} =r_{14}=r_{34}=\alpha <r_{12}$ and that the diagonals are
$r_{24}=r_{13}=\beta$. Then from de Dziobeck  equation $D=0$ we get
$$
\left(\beta ^3-\alpha ^3\right)^2 \left(\alpha ^3-r_{12}^3\right)=0.
$$
So either $r_{12}=r_{23} =r_{14}=r_{34}=\alpha$ and
$r_{24}=r_{13}=\beta$ which corresponds to a square, or
$\beta=\alpha$ which is not possible because it implies
$r_{24}=r_{13}=r_{12}=r_{23} =r_{14}=r_{34}=\alpha$ (see
Lemma~\ref{lemomega2}).

Proceeding in a similar way when the equal sides are $r_{12}
=r_{23}=r_{14}=\alpha >r_{34}$ the Dziobeck equation  becomes
$$
\left(\beta ^3-\alpha ^3\right)^2 \left(\alpha ^3-r_{34}^3\right)=0.
$$
When $r_{34}=\alpha$ we get again the square and when $\beta=\alpha$
we get condition $r_{24}=r_{13}=r_{12}=r_{23} =r_{14}>r_{34}$. This
condition can be satisfied only when the positions of $m_3$ and
$m_4$ coincide and the resulting configuration is an equilateral
triangle. Substituting the above relation into \eqref{eq:fin13} and
\eqref{eq:fin14} we get $m_3=m_4=0$.
\end{proof}

\section{The set of realizable trapezoid central
configurations}\label{si}

In this section we characterize the set of realizable trapezoid
central configurations.

\begin{proposition} \label{proponova}
The boundaries of $\widetilde{\Omega}'$ (see Lemma~\ref{lemomega2})
consist of a square, a curve corresponding to the rhombus, a curve
containing the isosceles trapezoids and a curve corresponding to
degenerate central configurations with $m_4=0$.
\end{proposition}

\begin{proof}
The possible boundaries of $\widetilde{\Omega}'$ are the sets where
either $r_{24}=r_{13}$, $r_{13}=r_{12}$, $r_{12}=r_{23}$,
$r_{23}=r_{14}$, or $r_{14}=r_{34}$. Next we characterize these
boundaries.

\noindent \emph{Case A: $r_{24}=r_{13}$}. The trapezoids having
equal diagonals are the rectangle, the square and the isosceles
trapezoid. The rectangle is not a realizable central configuration.

\noindent \emph{Case B: $r_{13}=r_{12}$}. After substituting
$r_{13}=r_{12}$ into equation $D=0$ we get the following three
subcases. Note that the configurations coming from this condition
are central configurations of the restricted problem; i.e. with one
or more masses equal to zero.

\begin{itemize}
\item[B.1:] $r_{12}=r_{14}$. This implies  $r_{12}=r_{23}=r_{14}$,
so the masses 
$m_1$, $m_2$, $m_3$ and $m_4$ are located at the vertices of an 
equilateral triangle with $r_{34}=0$.

\item[B.2:] $r_{12}=r_{23}$. In this case $r_{13}=r_{12}=r_{23}$,
this means that the masses $m_1$, $m_2$, and $m_3$ are at the
vertices of an equilateral triangle.

\item[B.3:]  $r_{24}=r_{34}$. This implies $r_{24}=r_{13}=r_{12}=
r_{23}=r_{14}=r_{24}$ which is  not possible.
\end{itemize}

\noindent \emph{Case C: $r_{12}=r_{23}$}. After substituting
$r_{12}=r_{23}$ into equation $D=0$ we get the following subcacases.

\begin{itemize}
\item[C.1:] $r_{13}=r_{23}$. Corresponds to case B.2.

\item[C.2:] $r_{23}=r_{24}$. This implies $r_{24}=r_{13}=r_{12}=r_{23}$,
so it corresponds also to case B.2.

\item[C.3:] $r_{14}=r_{34}$. In this case $r_{12}=r_{23}$ and $r_{14}=r_{34}$,
so the configuration is a kite. Since the configuration must be also
a trapezoid it is necessarily a rhombus.
\end{itemize}

\noindent\emph{Case D: $r_{23}=r_{14}$}. The trapezoids having two
equal sides are the isosceles trapezoid, the rhombus, the square,
the parallelogram and the rectangle. The last two do not correspond
to realizable central configurations (see Proposition~\ref{prop1n}).

\noindent \emph{Case E: $r_{14}=r_{34}$}. After substituting
$r_{14}=r_{34}$ into equation $D=0$ we get the following subcacases.

\begin{itemize}
\item[E.1:] $r_{12}=r_{23}$. Corresponds to case C.3.

\item[E.2:] $r_{13}=r_{34}$. This implies $r_{13}=r_{12}=r_{23}=
r_{14}=r_{34}$, so the configuration is a rhombus.

\item[E.3:]  $r_{24}=r_{34}$. Corresponds to case B.3.
\end{itemize}
\end{proof}

Next we give the shape of the set of realizable trapezoid central
configurations. To simplify the computations we parametrize the set
of realizable central configurations by using the positions of the
masses: $q_1=(0,0), q_2=(0,a), q_3=(1,b), q_4=(1,c)$, with $a\geq 1$
and $b> c$, and substituting the corresponding mutual distances into
the Dziobeck equation $D=0$. This equation gives a relation between
$a,b,c$ which provides an implicit $2$--dimensional surface in
$\mathbb{R}^3$.

\begin{theorem}
Let $q_1=(0,0), q_2=(0,a), q_3=(1,b), q_4=(1,c)$, with $a \geq 1$
and $b>c$, be the positions of the masses $m_1$, $m_2$, $m_3$ and
$m_4$ respectively, then the set $\Omega$ of realizable CC is
$$
\{(a,b,c)\in\mathbb{R}^{3}\ : \ a \geq 1,\ b > c, \ D=0,r_{24}\geq
r_{13}>r_{12}\geq r_{23}\geq r_{14}\geq r_{34}\},
$$
and the boundary of $\Omega$ is $\mathcal{C}_1\cup \mathcal{C}_2\cup
\mathcal{C}_3\cup P_1\cup P_2 \cup P_3$  where
\[
\begin{split}
\mathcal{C}_1&=\{(2/\sqrt{3},1/\sqrt{3},c)\::\:c\in(-1/\sqrt{3},1/\sqrt{3})\},\\
\mathcal{C}_2&=\{\left(\sqrt{1+c^2},c+\sqrt{1+c^2},c\right)\::\:c\in (-1/\sqrt{3},0)\},\\
\mathcal{C}_3&=\{(a,b,c)\ : \ a=b+c, f(b,c)=0, c\in (0,1/\sqrt{3})\},\\
P_1&=(2/\sqrt{3},1/\sqrt{3},1/\sqrt{3}),\\
P_2&=(2/\sqrt{3},1/\sqrt{3},-1/\sqrt{3}),\\
P_3&=(1,1,0),
\end{split}
\]
and
\[
\begin{split}
f(b,c)=& \left((b+c)^3-(c^2+1)^{3/2}\right)
\left((b^2+1)^{3/2}-(b-c)^{3}\right)-\\ &
\left((b^2+1)^{3/2}-(b+c)^3\right)
\left((c^2+1)^{3/2}-(b-c)^3\right).
\end{split}
\]
See Figure~\ref{figomega} for the plot of the set $\Omega$.

The points of $\mathcal{C}_1$ provide configurations of the
restricted problem with $m_1$, $m_2$ and $m_3$ in an equilateral
triangle; the points of $\mathcal{C}_2$ and $\mathcal{C}_3$ provide
rhombus and isosceles trapezoid central configurations,
respectively;  $P_1$ corresponds to an equilateral triangle
configuration of the restricted problem with a collision of $m_3$
and $m_4$ at one of the vertices;  $P_2$ corresponds to a
configuration of the restricted problem with the masses at the
vertices of a rhombus and such that the positions of $m_1$, $m_2$
and $m_3$ are the vertices of an equilateral triangle; and $P_3$
corresponds to the square central configuration.
\end{theorem}

\begin{proof}
Easily we can compute $r_{24}=\sqrt{1+(a-c)^2}$,
$r_{13}=\sqrt{1+b^2}$, $r_{12}=a$, $r_{23}=\sqrt{1+(a-b)^2}$,
$r_{14}=\sqrt{1+c^2}$, and $r_{34}=b-c$. Proposition~\ref{proponova}
gives the characterization of the central configurations on the
boundaries of $\Omega$. We prove the result by using  this
characterization and the parametrization $(a,b,c)$.

On the boundary with the masses $m_1$, $m_2$ and $m_3$ at the
vertices of an equilateral triangle we have $r_{13}=r_{12}=r_{23}$.
Solving the system of equations $r_{13}=r_{12}=r_{23}$, we get
$a=2/\sqrt{3}$ and $b=1/\sqrt{3}$. Substituting this solution into
$r_{23}$ and $r_{14}$ and imposing the condition $r_{23}\geq
r_{14}$, we get the condition $-1/\sqrt{3}\le c\le 1/\sqrt{3}$. It
is easy to check that the solution $(a,b,c)=(2/\sqrt{3},
1/\sqrt{3},c)$ with $c\in [-1/\sqrt{3}, 1/\sqrt{3}]$ satisfies
$r_{24}\geq r_{13} \geq r_{12} \geq r_{23}\geq r_{14}\geq r_{34}$.
So the set $\mathcal{C}_1\cup P_1\cup P_2$ belongs to the boundary
of $\Omega$. Moreover it is easy to check that the point $P_1$
correspond to an equilateral triangle configuration with the masses
$m_3$ and $m_4$ colliding at the corresponding vertex of the
triangle; and the point $P_2$ corresponds to a rhombus configuration
such that $m_1$, $m_2$ and $m_3$ are at the vertices of an
equilateral triangle.

On the rhombus configurations we have $r_{12}=r_{23}=r_{14}=r_{24}$.
Solving the system of equations $r_{12}=r_{23}=r_{14}=r_{24}$ we get
the solution $s(c)=(a,b)=(\sqrt{1+c^2},c+\sqrt{1+c^2})$. Imposing
that this solution satisfies $r_{24}\geq r_{13}>r_{12}$ we get the
condition $-1/\sqrt{3}< c\le 0$. So $\mathcal{C}_2$ belongs to the
boundary of $\Omega$. Moreover  $s(0)=(1,1)$ and
$s(-1/\sqrt{3})=(2/\sqrt{3},1/\sqrt{3})$. So the endpoints of
$\mathcal{C}_2$ are $P_2$ and $P_3$.

On the isosceles trapezoid, configurations $(a,b,c)$ are such that
$r_{23}=r_{14}$, $r_{24}=r_{13}$ and $D=0$. If $r_{23}=r_{14}$ and
$r_{24}=r_{13}$, then the Dziobeck equation $D=0$ becomes
$$
\left(r_{14}^3-r_{13}^3\right) \left(\left(r_{12}^3-r_{14}^3\right)
\left(r_{13}^3-r_{34}^3\right)-\left(r_{13}^3-r_{12}^3\right)
\left(r_{14}^3-r_{34}^3\right)\right)=0.
$$
If $r_{14}=r_{13}$, then $r_{24}=r_{13}=r_{12}=r_{23}=r_{14}$. This
corresponds to the point $P_1$ (see the proof of
Proposition~\ref{prop1n}). Assume now that $r_{14}\ne r_{13}$.
Solving system $r_{23}=r_{14}$, $r_{24}=r_{13}$ we get $a=b+c$. So
$r_{24}=r_{13}=\sqrt{b^2+1}$, $r_{12}= b+c$,
$r_{23}=r_{14}=\sqrt{c^2+1}$, and  $r_{34}=b-c$ and condition $D=0$
is equivalent to condition  $f(b,c)=0$. Cors and Roberts in
\cite{cors}, using a different parametrization, proved the existence
of a unique one-parameter family of isosceles trapezoid central
configurations which is characterized by a differentiable function
of one of the parameters in the parameter space. Moreover they prove
that the endpoints of this family are the square configuration and
the configuration consisting of an equilateral triangle with the
masses $m_3=m_4=0$ at one of the vertices. In our parametrization
the differentiable function of the parameter is the function $b(c)$
defined implicitly by $f(b,c)=0$ and the endpoints of the curve are
the points $P_1$ and $P_3$. Thus $\mathcal{C}_3$ is the last curve
in the boundary of $\Omega$ and it is a curve joining $P_1$ and
$P_3$, see Figure~\ref{figomega}.
\end{proof}

\begin{figure}
\unitlength=1cm
\begin{center}
\includegraphics[width=7cm]{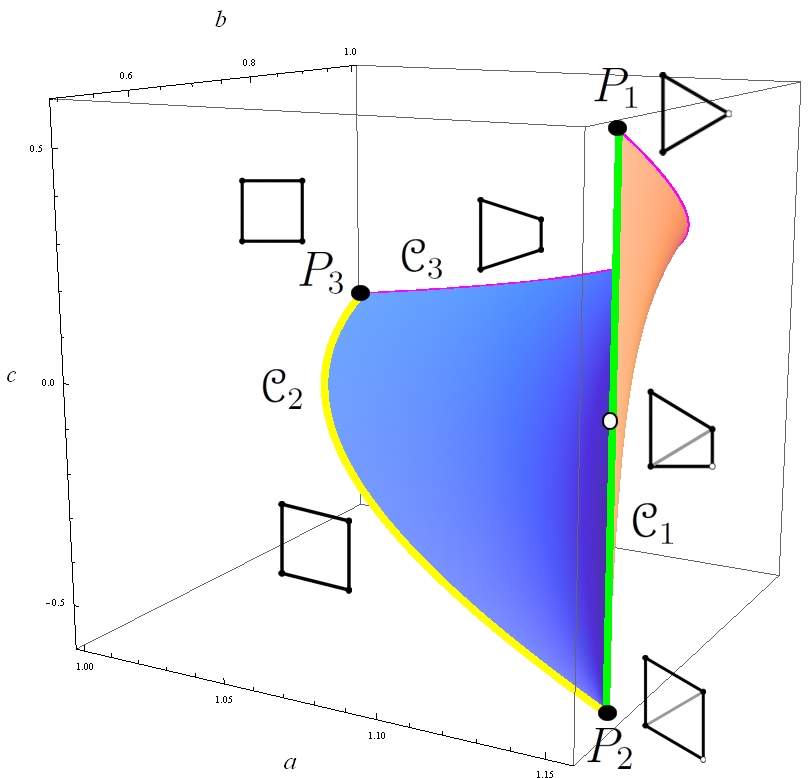}
\end{center}
\caption{The surface $\Omega$ of the trapezoid central
configurations in the $abc$--space.} \label{figomega}
\end{figure}

Unfortunately we are not able to prove that $\Omega$ is a
differentiable function over the two positions of the masses, as was
stablish in the co--circular case, see \cite{cors}. Nevertheless, in
the next section we prove that there exist a one--parameter family
of trapezoid central configurations that divides $\Omega$ in two
disjoint regions, namely the region that contains the trapezoid
central configurations and the one that contains the obtuse
trapezoid central configurations.

\section{The right trapezoid family}\label{right}

We suppose again that we are in the hypotheses of
Lemma~\ref{lemomega}; i.e. that  $r_{12}\geqslant r_{23}\geqslant
r_{14}\geqslant r_{34}$. We assume also that the position of the
masses $m_1=1, m_2, m_3$ and $m_4$ are respectively $(0,0), (0,a),
(1,b), (1,0)$ with $a\geqslant 1$ and  $a\geqslant b>0$. Easily we
can compute the mutual distances $r_{12}=a, r_{23} = \sqrt{1 +
(a-b)^2}, r_{14}=1, r_{34} = b,  $ and $ r_{13}=\sqrt{1+b^2}$ and
$r_{24} =\sqrt{1+a^2}$ (the diagonals).

First we give the set $\widetilde{\Omega}$ (see
Lemma~\ref{lemomega}) on the right trapezoid family parameterized by
the positions of the masses $(a,b)$. It is obvious that condition
$r_{24}>r_{12}$ is always satisfied and that condition
$r_{34}\leqslant r_{14}$ implies $b\leqslant 1$. On the other hand,
it is easy to see that conditions $r_{13}> r_{12}$ and
$r_{12}\geqslant r_{23}$ imply that $a< a_1(b)= \sqrt{1+b^2}$ and
$a\geqslant a_2(b)=(b^2+1)/(2b)$ respectively. From here we get
condition $a_1(b)> a_2(b)$ which is satisfied for $b> 1/\sqrt{3}$.
In short, the set $\widetilde{\Omega}$ (see Lemma~\ref{lemomega}) on
the right trapezoid family is
$$
\Omega_r=\left\{(a,b,c)\in\mathbb{R}^3\::\:\frac{b^2+1}{2b}\leqslant
a< \sqrt{1+b^2}, \quad \frac{1}{\sqrt{3}}< b\leqslant 1,
\quad c=0\right\}.
$$
We can see that $a_1$ is a decreasing function in
$b\in(1/\sqrt{3},1]$ with
 $a_1(1/\sqrt{3})=2/\sqrt{3}$ and $a_1(1)=1$, whereas $a_2$ is an
increasing function in $b\in(1/\sqrt{3},1]$ with
$a_2(1/\sqrt{3})=2/\sqrt{3}$ and $a_2(1)=\sqrt{2}$ (see
Figure~\ref{figrt}). Therefore $a\in(1,\sqrt{2})$.

\begin{figure}
\includegraphics[width=5cm]{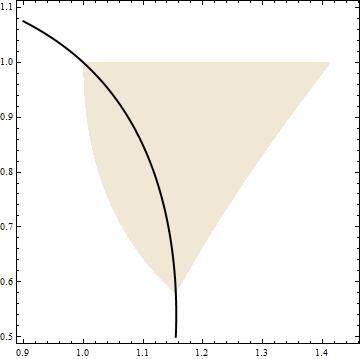} \hspace*{0.5cm}
\includegraphics[width=6.5cm]{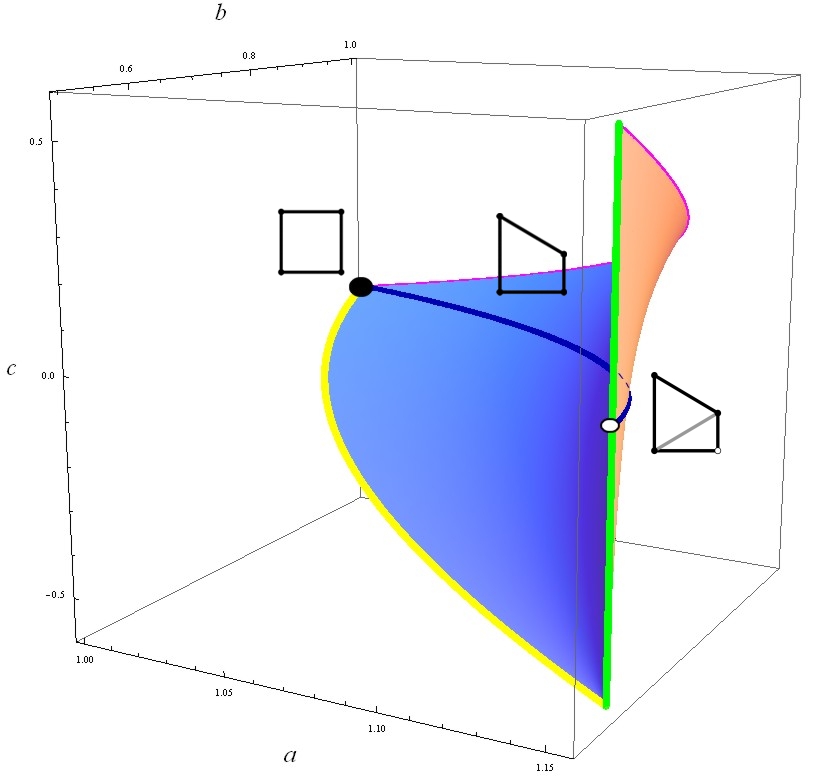}\\
(a) \hspace*{6cm} (b) \caption{(a) The region $\Omega_r$ and the the
curve $D=0$ for the right trapezoid family in the $a-b$ plane. (b)
The the curve $D=0$ for the right trapezoid family on the region
$\Omega$.} \label{figrt}
\end{figure}

\begin{theorem}
\label{theort} The curve $D=0$ is a graph with respect to the
variable $a$ in the region $\Omega_r$ (see Figure~\ref{figrt}). In
fact, $\partial D/\partial a$ evaluated at the curve $D=0$
restricted to $\Omega_r$ is negative.
\end{theorem}

\begin{proof}
When $b=1$  equation $D=0$ becomes
$$
\left(1-2 \sqrt{2}\right) \left(\left(a^2+1\right)^{3/2}-1\right)
\left(a^3-\left(a^2-2 a+2\right)^{3/2}\right)=0,
$$
which has a unique real solution with $a\ge 1$, the solution $a=1$.
After straightforward computations we see that substituting
$a=a_2(b)$ into $D$ we get a function of $b$ that is zero at $b=1$
and $b=1/\sqrt{3}$ and positive for $b\in(1/\sqrt{3},1)$. In a
similar way substituting $a=a_1(b)$ into $D$ we get a function of
$b$ that is zero at $b=0$ and $b=1/\sqrt{3}$ and negative for
$b\in(1/\sqrt{3},1]$. Therefore each curve in $\Omega_r$ joining a
point of the curve $a_1(b)$ with a point of the curve $a_2(b)$ has
at least a point with $D=0$. Therefore there exist al least one set
in $\Omega_r$  satisfying $D=0$. Next we see that this set is a
graph in the variable $a$ that joins the points $(1,1)$ and
$(2/\sqrt{3},1/\sqrt{3})$, see Figure~\ref{figrt}(a).

By simple computations we get
$$
\frac{\partial D}{\partial a}=\frac{\partial D}{\partial
r_{12}}\frac{\partial r_{12}}{\partial a}+\frac{\partial D}{\partial
r_{23}}\frac{\partial r_{23}}{\partial a}+\frac{\partial D}{\partial
r_{24}}\frac{\partial r_{24}}{\partial a},
$$
where
\begin{eqnarray*}
\frac{\partial r_{12}}{\partial a}&=& 1,\\
\frac{\partial r_{23}}{\partial a}&=&\frac{a-b}{r_{23}}=\frac{r_{12}-r_{34}}{r_{23}},\\
\frac{\partial r_{24}}{\partial
a}&=&\frac{a}{r_{24}}=\frac{r_{12}}{r_{24}},
\end{eqnarray*}
and
\begin{eqnarray*}
\frac{\partial D}{\partial r_{12}}&=&-3 r_{12}^2
\left(r_{13}^3-r_{23}^3\right) \left(r_{24}^3-r_{34}^3\right)-3
r_{12}^2 \left(r_{24}^3-r_{14}^3\right)
\left(r_{23}^3-r_{34}^3\right),\\
\frac{\partial D}{\partial r_{23}}&=&3 r_{23}^2
\left(r_{13}^3-r_{12}^3\right) \left(r_{24}^3-r_{14}^3\right)+3
r_{23}^2 \left(r_{12}^3-r_{14}^3\right)
\left(r_{24}^3-r_{34}^3\right),\\
\frac{\partial D}{\partial r_{24}}&=&3 r_{24}^2
\left(r_{13}^3-r_{12}^3\right) \left(r_{23}^3-r_{34}^3\right)-3
r_{24}^2 \left(r_{12}^3-r_{14}^3\right)
\left(r_{13}^3-r_{23}^3\right).
\end{eqnarray*}
Rearranging the terms in a convenient way $\partial D/\partial a$
can be written in terms of the mutual distances as
$$
\frac{\partial D}{\partial a}=3(f_1+f_2+f_3+f_4+f_5+f_6)
$$
where
\begin{eqnarray*}
f_1&=&-r_{12}^2 \left(r_{24}^3-r_{14}^3\right)
\left(r_{23}^3-r_{34}^3\right),\\
f_2&=&-r_{12}^2 \left(r_{13}^3-r_{23}^3\right)
\left(r_{24}^3-r_{34}^3\right),\\
f_3&=&r_{23} (r_{12}-r_{34}) \left(r_{13}^3-r_{12}^3\right)
\left(r_{24}^3-r_{14}^3\right),\\
f_4&=&r_{23} (r_{12}-r_{34}) \left(r_{12}^3-r_{14}^3\right)
\left(r_{24}^3-r_{34}^3\right),\\
f_5&=&r_{12} r_{24} \left(r_{13}^3-r_{12}^3\right)
\left(r_{23}^3-r_{34}^3\right),\\
f_6&=&-r_{12} r_{24} \left(r_{12}^3-r_{14}^3\right)
\left(r_{13}^3-r_{23}^3\right).
\end{eqnarray*}
Next, we will see that at the points of $\Omega_r$ satisfying
$D=0$ the following conditions hold: $f_1+f_4<0$,
$f_2+f_5<0$ and $f_3+f_6\leqslant 0$. Therefore $\partial D/\partial
a$ evaluated at the curve $D=0$ restricted to $\Omega_r$ is
negative.

{From} $D=0$ we get
$$
\left(r_{24}^3-r_{14}^3\right)
\left(r_{23}^3-r_{34}^3\right)=\frac{\left(r_{12}^3-r_{14}^3\right)
\left(r_{13}^3-r_{23}^3\right)
\left(r_{24}^3-r_{34}^3\right)}{r_{13}^3-r_{12}^3}.
$$
Then  $f_1+f_4$ can be written as
$$
f_1+f_4= \frac{g_1 \left(r_{12}^3-r_{14}^3\right)
\left(r_{24}^3-r_{34}^3\right)}{r_{13}^3-r_{12}^3},
$$
where
$$
g_1=-r_{34} r_{23} \left(r_{13}^3-r_{12}^3\right)-r_{12}
(r_{12}-r_{23})
   \left(r_{12} r_{23}
   (r_{12}+r_{23})+r_{13}^3\right).
  $$
Since $r_{13}>r_{12}\geqslant r_{23}$ on $\Omega_r$ and $g_1<0$ on
$\Omega_r$, it follows that $f_1+f_4<0$ evaluated at the curve $D=0$
restricted to $\Omega_r$.

Proceeding as above, using that $D=0$ expression $f_2+f_5$ can be
written as
$$
f_2+f_5=\frac{g_2 \left(r_{13}^3-r_{12}^3\right)
\left(r_{23}^3-r_{34}^3\right)}{r_{12}^3-r_{14}^3},
$$
where
$$
g_2=-r_{12}(r_{24}-r_{12})(r_{14}^3+r_{12}^2r_{24}+r_{12}r_{24}^2).
$$
Since $r_{24}>r_{12}$ on $\Omega_r$,  $g_2<0$ on $\Omega_r$, we
obtain that $f_2+f_5<0$ evaluated at the curve $D=0$ restricted to
$\Omega_r$.

Finally, using that $D=0$ expression $f_3+f_6$ can be written as
$$
f_3+f_6=\frac{g_3 \left(r_{13}^3-r_{12}^3\right)
\left(r_{24}^3-r_{14}^3\right)}{r_{24}^3-r_{34}^3},
$$
where
$$
g_3=-r_{24}r_{12}(r_{23}^3 - r_{34}^3) + r_{23}(r_{12} -
r_{34})(r_{24}^3 - r_{34}^3).
$$
There exists a curve in $\Omega_r$ such that $g_3=0$, so the
previous arguments are not valid to prove that $f_3+f_6$ evaluated
at the curve $D=0$ restricted to $\Omega_r$ is negative.

In order to avoid the last obstacle, by using resultants we will prove 
that there are no values
$(a,b)$ in the interior of $\Omega_r$ for which $D$ and $g_3$ are
zero simultaneously.

Let $\mathrm{Res}[P,Q,x]$ denote the resultant of the polynomials
$P(x,y)$ and $Q(x,y)$ with respect to $x$. The resultant
$\mathrm{Res}[P,Q,x]$ is a polynomial in the variable $y$ satisfying
the following property: if $(x,y)=(x^*,y^*)$ is a solution of system
$P(x,y)=0$, $Q(x,y)=0$ then $y=y^*$ is a zero of
$\mathrm{Res}[P,Q,x]$. In other words, the set of zeroes of
$\mathrm{Res}[P,Q,x]$ contains the components $y$ of all 
solutions of the system $P(x,y)=0$ and $Q(x,y)=0$. We observe that it may contain
additional solutions that are not related with the solutions of the
system $P(x,y)=0$ and $Q(x,y)=0$.

Let
\begin{equation}
\label{eqres}
\begin{array}{l}
e_1=e_1(a,b,r_{13},r_{23},r_{24})=0,\\
e_2= e_2(a,b,r_{23},r_{24})=0,
\end{array}
\end{equation}
be the system defined by the two equations $D=0$ and $g_3=0$ after
performing the substitutions $r_{12}=a$, $r_{14}=1$ and $r_{34}=b$.
Here we think that the mutual distances $r_{13}$, $r_{23}$ and
$r_{24}$ are the positive solutions of system
\begin{eqnarray*}
e_3=e_3(a,b,r_{13})&=&r_{13}^2-(b^2+1)=0,\\
e_4=e_4(a,b,r_{23})&=&r_{23}^2-((a-b)^2+1)=0,\\
e_5=e_5(a,b,r_{24})&=&r_{24}^2-(a^2+1)=0.
\end{eqnarray*}
Using resultants we eliminate the variables $r_{13}$, $r_{23}$ and
$r_{24}$ from the equations \eqref{eqres} in the following way. We
eliminate the variable $r_{13}$ from $e_1$ by doing the resultant
$$
R_1=\mathrm{Res}[e_1,e_3,r_{13}].
$$
Then we eliminate the variable $r_{23}$ from $R_1$ and $e_2$ by
doing the resultants
$$
S_1=\mathrm{Res}[R_1,e_4,r_{23}],\qquad
S_2=\mathrm{Res}[e_2,e_4,r_{23}],
$$
and the variable $r_{24}$ from $S_1$ and $S_2$ by doing the
resultants
$$
T_1=\mathrm{Res}[S_1,e_5,r_{24}],\qquad
T_2=\mathrm{Res}[S_2,e_5,r_{24}].
$$
Here $T_1=16a^2b^2\widetilde{T}_1$ and $T_2=b^4\widetilde{T}_2$,
were $\widetilde{T}_1$ and $\widetilde{T}_2$ are polynomials of
total degree 64 and 16, respectively, in the variables $a$ and $b$.
Note that by the properties of resultants the set of solutions of
the new system of equations $\widetilde{T}_1=0$, $\widetilde{T}_2=0$
contains all solutions with $a,b\ne 0$ of system
\eqref{eqres}, or equivalently all the solutions with $a,b\ne 0$ of
system $D=0$, $g_3=0$ (thinking $D$ and $g_3$ as a function of $a,b$
via the mutual distances $r_{ij}$).

Now we solve system $\widetilde{T}_1=0$, $\widetilde{T}_2=0$ by
using resultants again. We compute
$\mathrm{Res}[\widetilde{T}_1,\widetilde{T}_2,a]$ and we get the
polynomial
\begin{equation}
\begin{array}{rcl}
W(b)&=& (b-1)^{16} b^{96} \left(b^2+1\right)^{20}
\left(b^2-b+1\right)^4 \left(b^2+b+1\right)^{16} \\ & &\left(3
b^2-4\right)^2 \left(21 b^4+12 b^2+16\right)^4 \left(21 b^4+24
b^2+16\right)^2 \\ & & W_1(b)\,W_2(b)\,W_3(b)\,W_4(b),
\end{array}
\label{eqres2}
\end{equation}
where $W_1$, $W_2$, $W_3$ and $W_4$ are polynomials of degrees 162,
202, 210, and 214 respectively. From properties of resultants the
set of zeroes of $W$ contains the component $b$ of all the solutions
of system $\widetilde{T}_1=0$, $\widetilde{T}_2=0$. Recall that we
are only interested in solutions belonging to $\Omega_r$, so we only
consider zeroes with $1/\sqrt{3}<b\leqslant 1$. We compute
analytically the zeroes of the first eight factors of $W$ and
numerically the zeroes of the remaining four factors of $W$ and we
get the following solutions with $1/\sqrt{3}<b\leqslant 1$
\begin{eqnarray*}
b&=0.61283303\dots, \quad &b=0.69216326\dots, \\
b&=0.71614387\dots,
&b=0.76874157\dots, \\ b&=0.79099409\dots, & b=0.82966657\dots, \\
b&=0.91953907\dots, &b=1.
\end{eqnarray*}
Next we compute $\mathrm{Res}[\widetilde{T}_1,\widetilde{T}_2,b]$
and we get the polynomial
\begin{equation}
\begin{array}{rcl}
w(a)&=& (a-1)^{24} a^{64} \left(a^2+1\right)^{40}
\left(a^2-a+1\right)^4 \left(a^2+a+1\right)^8 \\ & & \left(3
a^2-1\right)^2 \left(21 a^4+6 a^2+1\right)^4 \left(21 a^4+54
a^2+49\right)^2 \\ & & w_1(a)\,w_2(a)\,w_3(a)\,w_4(a),
\end{array}
\label{eqres21}
\end{equation}
where $w_1$, $w_2$, $w_3$ and $w_4$ are polynomials of degrees 162,
202, 210, and 214 respectively. The set of zeroes of $w$ contains
the component $a$ of all the solutions of system
$\widetilde{T}_1=0$, $\widetilde{T}_2=0$. Since we are only
interested in solutions belonging to $\Omega_r$, we only consider
zeroes with $1\leqslant a\leqslant \sqrt{2}$. As above we compute
analytically the zeroes of the first eight factors of $w$ and
numerically the zeroes of the remaining four factors of $w$ and we
get the following solutions with $1\leqslant a\leqslant \sqrt{2}$
\begin{eqnarray*}
a&=1,\phantom{07124596\dots} \ \qquad &a=1.04304633\dots,\\
a&=1.07124596\dots, \qquad &a=1.08484650\dots,\\
a&=1.09217286\dots,\qquad &a=1.10559255\dots,\\
a&=1.16459040\dots.\qquad &
\end{eqnarray*}
We consider $D$ and $g_3$ as a function of $(a,b)$ by substituting
the expressions of the mutual distances $r_{ij}$. The possible
solutions $(a,b)$ of system $D=0$, $g_3=0$ are the pairs
$(a,b)=(a^*,b^*)$ formed by a zero $a^*$ of $w$ with
$a^*\in[1,\sqrt{2}]$,  and  a zero $b^*$ of $W$ with
$b^*\in(1/\sqrt{3},1]$. Substituting all possible pairs
$(a,b)=(a^*,b^*)$ into $D=0$, $g_3=0$ we see that $(a,b)=(1,1)$ is
the unique pair that provides a solution of the system. This point
belongs to the boundary of $\Omega_r$. Therefore the function $g_3$
does not change its sign on the solutions of $D=0$ that belong to
$\Omega_r$.

It is easy to the check that the point
$(a,b)=(2/\sqrt{3},1/\sqrt{3})$, which belongs to the boundary of
$\Omega_r$, satisfies $D=0$. Moreover the function $g_3$ evaluated
at $(a,b)=(2/\sqrt{3},1/\sqrt{3})$ is negative. Hence $g_3\leqslant
0$ in $\Omega_r$. This end the proof of the theorem.
\end{proof}

In short, the set of realizable right trapezoid central
configurations is
$$
\Omega_r'=\{(a,b,c)\in\Omega_r\::\: D=0\},
$$
and it is plotted in Figure~\ref{figrt}. In
Figure~\ref{massestrapezir} we plot the masses along the right
trapezoid family parameterized by the parameter $a$. We note that
the limit case $(a,b)=(1,1)$ correspond to the square with equal
masses, and the limit case $(a,b)=(2/\sqrt{3},1/\sqrt{3})$
corresponds to a right trapezoid central configurations with masses
$m_1=1$, $m_4=0$ and
\begin{equation}
\begin{split}
m_2=&\frac{7 \left(8 \sqrt{3}-9\right) \left(49+8
\sqrt{7}\right)}{2511}= 0.94993335\dots,\\[5pt]
m_3=&\frac{2}{63} \left(8 \sqrt{3}-9\right)= 0.15417163\dots,
\label{massestr}
\end{split}
\end{equation}
such that the masses $m_1,m_2,m_3$ form an equilateral triangle with
edge length $2/\sqrt{3}$.

\begin{figure}
\includegraphics[width=5cm]{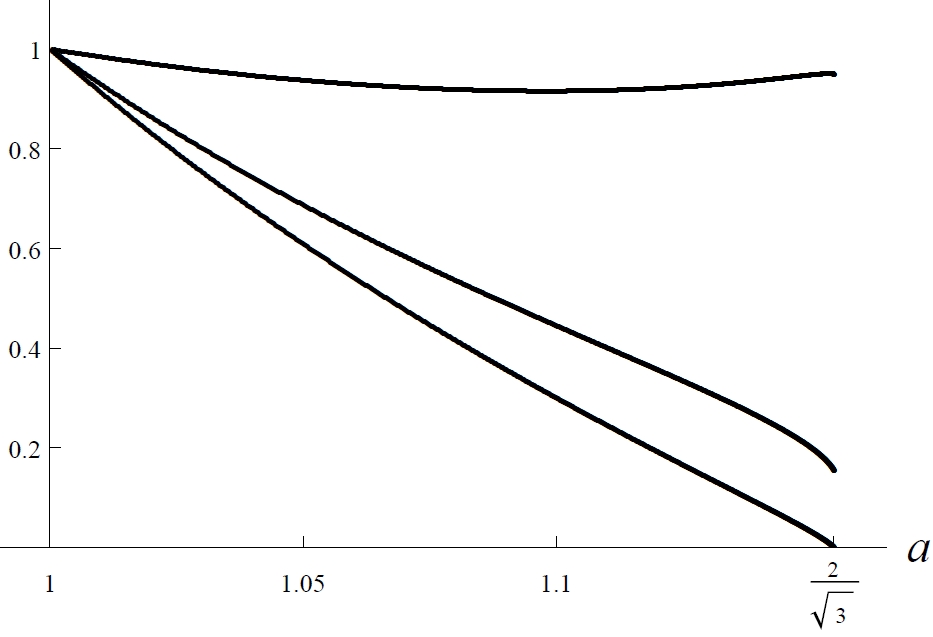}
\caption{The plot of the masses $m_2$ (upper line), $m_3$ (medium
line) and  $m_4$ (bottom line) along the right trapezoid family.}
\label{massestrapezir}
\end{figure}

\section{Trapezoid CC with a couple of equal masses}\label{masses}

In this section we will study the trapezoid CC with a pair of equal
masses. In \cite{cors} Cors and Roberts shown that for a given order
of the mutual distances in any co--circular central configuration
the set of masses is completely ordered.  A similar result for the
trapezoid central configurations has been obtained recently by
Santoprete \cite{santoprete}. Although in that case the masses are
not totally ordered. With our particular choice of labeling, from
\cite{santoprete} any trapezoid central configuration satisfies
\begin{equation}
m_4\leq m_3\leq m_1=1 \qquad m_4\leq m_2.
\end{equation}

Moreover, also from Santoprete \cite{santoprete}, we know that if
$m_3=m_1=1$ or $m_2=m_4$, then the central configuration is a
rhombus and the remaining two masses have to be equal. And if
$m_3=m_4$, then the central configuration is an isosceles trapezoid
and again the two remaining masses are necessarily equal.
Figure~\ref{figmasses2} shows the full set of masses for which a
trapezoid central configuration exist.

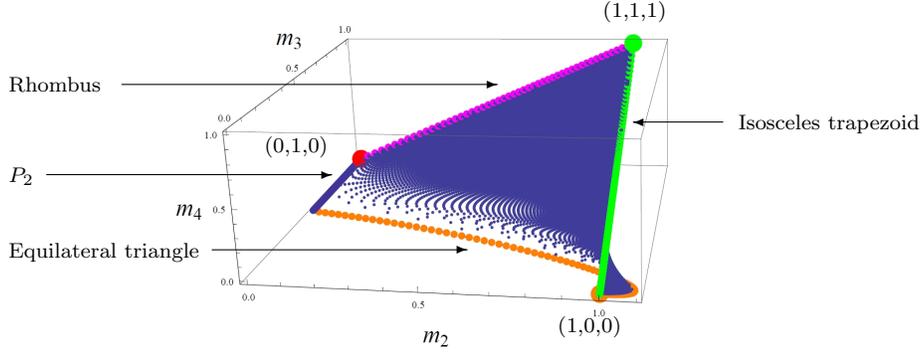
\begin{figure}
\unitlength=1cm
\begin{picture}(12,5)
\put(2,0){\includegraphics[width=7cm]{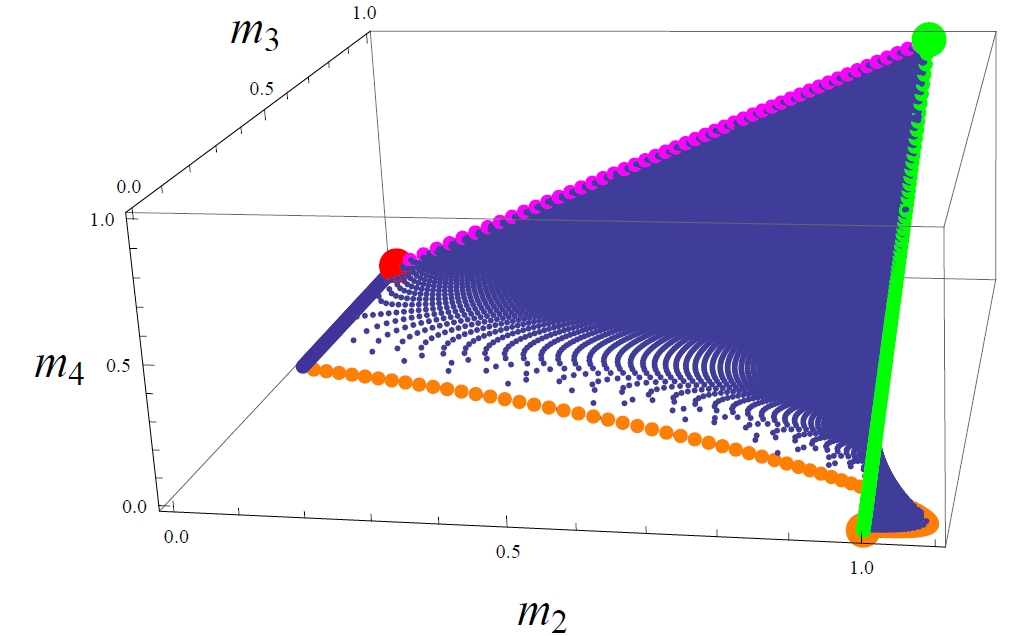}}
\put(2.7,1.3){\vector(1,0){3.4}} \put(0.5,2.3){\vector(1,0){3.8}}
\put(1.5,3.5){\vector(1,0){5}} \put(9.5,3){\vector(-1,0){1.2}}
\put(3.4,2.6){\tiny{(0,1,0)}} \put(7.9,4.4){\tiny{(1,1,1)}}
\put(7.3,0.2){\tiny{(1,0,0)}} \put(0,1.2){\tiny{Equilateral
triangle}} \put(0,2.2){\tiny{$P_2$}} \put(0,3.4){\tiny{Rhombus}}
\put(9.7,2.9){\tiny{Isosceles trapezoid}}
\end{picture}
\caption{The image of $\Omega$ in $m_2 m_3 m_4$--space under
equations (\ref{eq:fin12}), (\ref{eq:fin13}) and (\ref{eq:fin14})
with $m_1=1$. } \label{figmasses2}
\end{figure}

{From} the previous results only two cases of trapezoid central
configuration with only a pair of equal masses remains unknown,
namely, $m_2=m_3$ and $m_2=m_1=1$. In the next two subsections we
are going to show the existence of these two classes of trapezoid
central configuration. Something remarkable is that we will proved
analytically the existence of non--symmetric trapezoid central
configurations with two equal masses. As far as we know this result
was known numerically, but we think that is the first time that this
result is proved analytically in the four--body problem.

Now we study the value of masses along the boundary of $\Omega$.

\emph{The equilateral triangle family}.  By substituting the points
of $\mathcal{C}_1$ into \eqref{eq:fin12}, \eqref{eq:fin13} and
\eqref{eq:fin14} we get $m_4=0$ and
\begin{equation}
\label{m2onc1} m_2=\mu_2(c)=\frac{\left(8 \sqrt{3}-9 k_1^3\right)
k_2^3}{9 k_1^3 \left(k_2^3-8\right)}, \quad m_3=\mu_3(c)=\frac{2
\left(\sqrt{3}-3 c\right)^2 \left(8 \sqrt{3}-9 k_1^3\right)}{27
\left(\sqrt{3} c+1\right) k_1^3 k_2^2},
\end{equation}
where $k_1=\sqrt{c^2+1}$ and $k_2=\sqrt{3 c^2-4 \sqrt{3} c+7}$. The
function $\mu_2$ is defined  for all $c\in[-1/\sqrt{3},1/\sqrt{3})$,
$\mu_2(-1/\sqrt{3})=0$ and $\mu_2(c)\to 1$ when $c\to 1/\sqrt{3}$;
moreover it is increasing for $c\in[-1/\sqrt{3},c_0)$, it has a
maximum at $c_0= 0.27448350\dots$ with $\mu_2(c_0)=1.0912476\dots$
and it is decreasing for $c\in (c_0,1/\sqrt{3})$. The function
$\mu_3$ is defined  for all $c\in(-1/\sqrt{3},1/\sqrt{3}]$,
$\mu_3(1/\sqrt{3})=0$, $\mu_3(c)\to 1/2$ when $c\to -1/\sqrt{3}$ and
it is decreasing for $c\in(-1/\sqrt{3},1/\sqrt{3}]$. The plot of the
masses on $\mathcal{C}_1$ is given in Figure~\ref{massesvora} (a).

\emph{The rhombus family}. By substituting the points of
$\mathcal{C}_2$ into \eqref{eq:fin12}, \eqref{eq:fin14} and
\eqref{eq:squa} we get $m_1=m_3=1$ and
$$
m_2=m_4=\mu_r(c)=-\frac{\left((c-k)^2+1\right)^{3/2}
\left(\left((c+k)^2+1\right)^{3/2}-k^3\right)}{\left(k^3-
\left((c-k)^2+1\right)^{3/2}\right) \left((c+k)^2+1\right)^{3/2}},
$$
where $k=\sqrt{1+c^2}$. Note that on the rhombus family expression
\eqref{eq:fin13} is not well defined and we should take expression
\eqref{eq:squa} instead of it.  We can see that $\mu_r$ is an
increasing function defined for all $c\in[-1/\sqrt{3},0]$, such that
$\mu_r(-1/\sqrt{3})=0$ and $\mu_r(0)=1$. The plot of the masses on
$\mathcal{C}_2$ is given in Figure~\ref{massesvora} (b).

\emph{The isosceles trapezoid family}. On the isosceles trapezoid
family we know that $m_1=m_2=1$ and $m_3=m_4=\mu$ (see for instance
\cite{cors}), but since we do not have an explicit expression of the
solutions of $f(b,c)=0$, we cannot give the explicit expression of
$\mu$ as a function of the parameter $c$. Studying numerically the
function $\mu$  we see that it is a decreasing function in $c\in
(0,1/\sqrt{3})$ such that $\mu\to 1$ when $c\to 0$ and $\mu\to 0$
when $c\to 1/\sqrt{3}$. The plot of the masses on $\mathcal{C}_3$ is
given in Figure~\ref{massesvora} (c).

Note that if we approach to $P_2$ over the set $\mathcal{C}_2$ then
$m_3\to 1$, whereas if we approach to $P_2$ over the set
$\mathcal{C}_1$ then  $m_3\to 1/2$. Thus the limit of $m_3$ as we
approach to $P_2$ depends on the path you take and $m_3$ has a non
removable discontinuity at $P_2$.

\begin{figure}
\begin{tabular}{ccc}
\includegraphics[width=3.5cm]{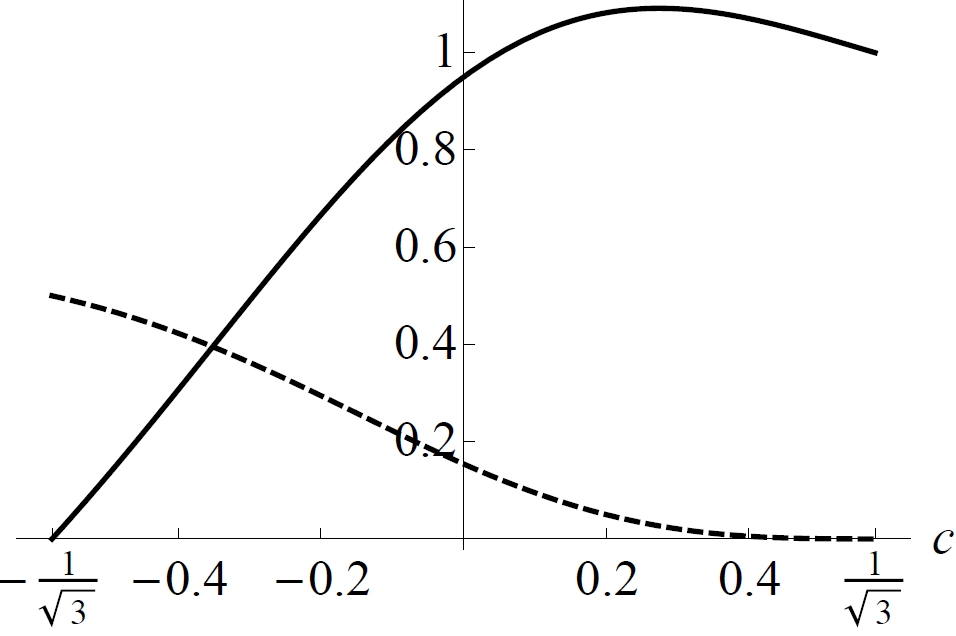} & \includegraphics[width=3.5cm]
{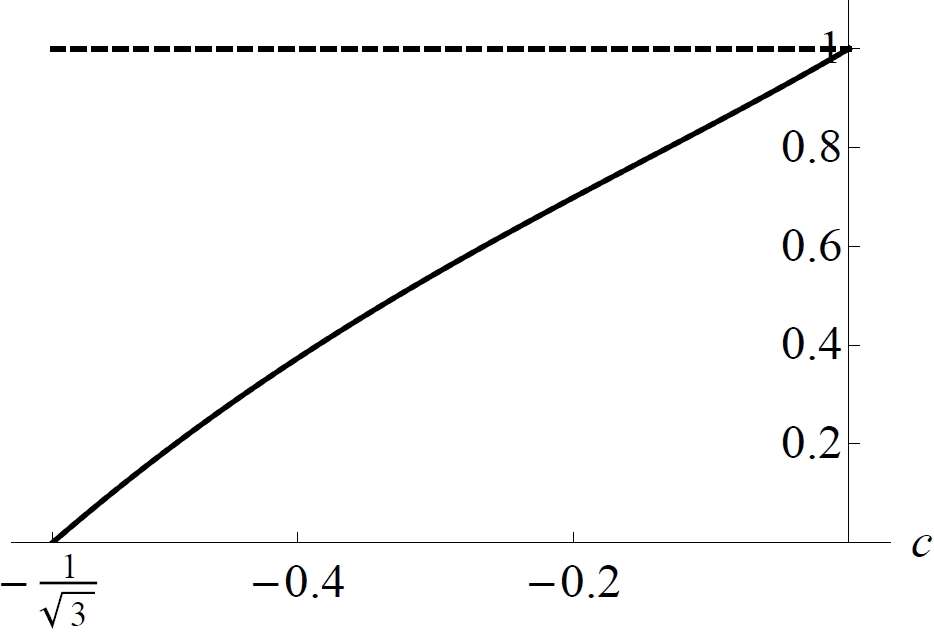} & \includegraphics[width=3.5cm]{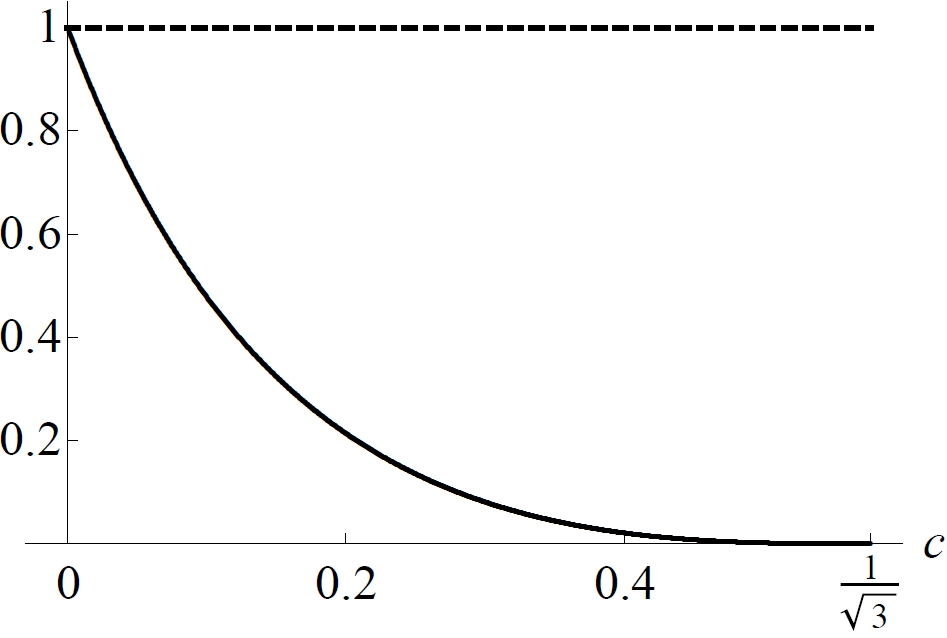}\\
(a) & (b) & (c)
\end{tabular}
\caption{(a) The plot of the masses $m_2$ (continuous line) and
$m_3$ (dashed line) on the boundary $\mathcal{C}_1$. (b) The plot of
the masses $m_2=m_4$ (continuous line) and $m_1=m_3=1$ (dashed line)
along the rhombus family. (c) The plot of the masses $m_1=m_2=1$
(dashed line) and $m_3=m_4$ (continuous line) along the isosceles
trapezoid family.  } \label{massesvora}
\end{figure}

\subsection{$m_2=m_3$}

Take $f=m_2-m_3$. We are interested in the solutions of $f=0$. On
$\mathcal{C}_2$ (corresponding to the rhombus family) we have
$r_{12} = r_{23} = r_{34} = r_{14}$ and $m_1=m_3=1$.

We know from equation (\ref{eq:fin12}) that on $\mathcal{C}_2$
\begin{equation}
m_2  =  \frac{  r_{23}^3 r_{24}^3 \, (r_{13}^3 - r_{14}^3)}{
r_{13}^3 r_{14}^3 \, (r_{24}^3 - r_{23}^3)} = \frac{ r_{24}^3 \,
(r_{13}^3 - r_{12}^3)}{ r_{13}^3 \, (r_{24}^3 - r_{12}^3)}.
\end{equation}
Therefore
\begin{eqnarray*}
m_2 < 1 &\iff&   r_{24}^3 \, (r_{13}^3 - r_{12}^3)<r_{13}^3 \, (r_{24}^3 - r_{12}^3 )\\
&\iff&- r_{24}^3 \, r_{12}^3<-r_{13}^3 \,  r_{12}^3  \iff r_{24}^3 >
r_{13}^3.
\end{eqnarray*}
The last inequality follows from the fact that $r_{24} > r_{13} >
1$. Therefore on $\mathcal{C}_2$ we have $f=m_2-m_3<0.$

Now on $\mathcal{C}_3$ (corresponding to the isosceles trapezoid
family) we have $ r_{23} = r_{14}$, $ r_{24} = r_{13}$ and
$m_1=m_2=1$. We also know that on this family $m_3=m_4$.

\begin{eqnarray*}
m_3=m_4 < 1  &\iff&   \frac{  r_{34}^2 \, (r_{13}^3 - r_{12}^3)}{
r_{12}^2  \, (r_{13}^3 - r_{34}^3)} \\
&\iff& r_{34}^2 (r_{13}^3 - r_{12}^3) < r_{12}^2 (r_{13}^3 - r_{34}^3)\\
&\iff& -r_{13}^3 (r_{12}^2 - r_{34}^2) < r_{12}^2 r_{34}^2 (r_{12} - r_{34})\\
&\iff& -r_{13}^3 (r_{12} - r_{34})  (r_{12} + r_{34})< r_{12}^2 r_{34}^2 (r_{12} - r_{34})\\
&\iff& -r_{13}^3  (r_{12} + r_{34})< r_{12}^2 r_{34}^2.
\end{eqnarray*}
That is, on $\mathcal{C}_3$ we have $f=m_2-m_3 > 0.$ Therefore for
any given path connecting $\mathcal{C}_2$ and $\mathcal{C}_3$ in
$\Omega$, there exist $(a,b,c)\in \Omega$ such that $f=0$ or
equivalently $m_2=m_3$.

Numerically we show that for any fixed path connecting
$\mathcal{C}_2$ and $\mathcal{C}_3$ in $\Omega$ the solution of
$f=0$ is unique. Curve representing the zeros of $f$ in $\Omega$
goes from $P_3$ to $(2/\sqrt{3},1/\sqrt{3},-0.351839354\dots) \in
\mathcal{C}_1$ (see Figure \ref{fig23}). To verify the last
statement, we observe that the function $f$ on $\mathcal{C}_1$
becomes
$$
f= \frac{\left(\frac{8}{3 \sqrt{3}}-\left(c^2+1\right)^{3/2}\right)
\left(-\frac{6 \left(\frac{1}{\sqrt{3}}-c\right)^2}{8-\left(3 c^2-2
\sqrt{3} c+1\right)^{3/2}}-\frac{\left(3 c^2-4 \sqrt{3}
c+7\right)^{3/2}}{8-\left(3 c^2-4 \sqrt{3}
c+7\right)^{3/2}}\right)}{\left(c^2+1\right)^{3/2}}.
$$
We apply Sturm Theorem to conclude that $f=0$, as a polynomial of
degree 24, has a unique real solution in
$c\in(-1/\sqrt{3},1/\sqrt{3})$, namely $c=-0.351839354\dots$

\begin{figure}
\includegraphics[width=5cm]{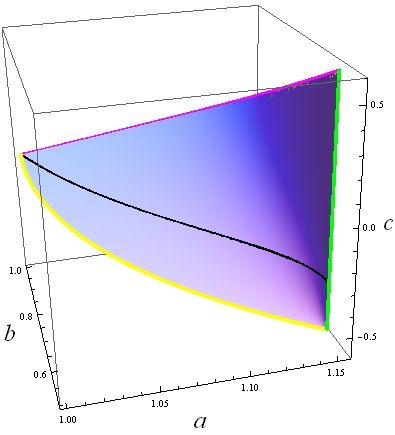} \hspace*{0.5cm}
\includegraphics[width=5.5cm]{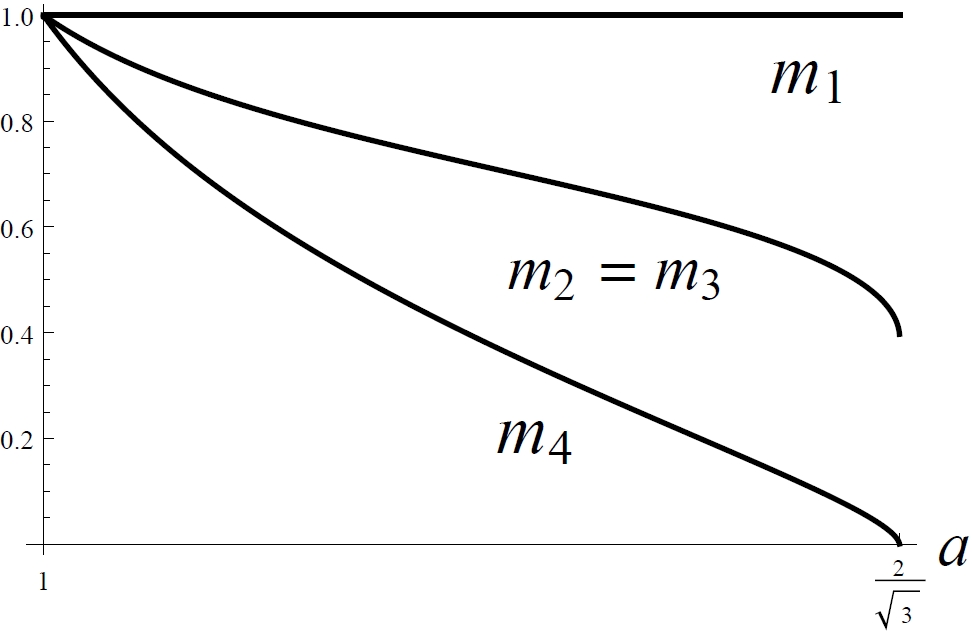}\\
\caption{(a) In black, the curve $m_2=m_3$ on the region $\Omega$.
(b) The curves $m_1=1$, $m_2=m_3$, and $m_4$ parametrized by $a$ on
the curve in (a).} \label{fig23}
\end{figure}

\subsection{$m_2=m_1=1$}

It is clear that $m_2=m_1=1$ along the isosceles trapezoid central
configuration family, that is, on the boundary $\mathcal{C}_3$.
Moreover $m_1=m_2=m_3=m_4=1$ on $P_3$, that is, on the square
configuration, and  $m_2- m_1<0$ on  $\mathcal{C}_2\cup P_2$.

On the boundary $\mathcal{C}_1 \cup \{ P_1, P_2 \}$ we have $m_4=0$
and $m_1, m_2$ and $m_3$ at the vertices of an equilateral triangle.
Let  $m_2=\mu_2(c)$ be the value of the mass $m_2$ on
$\mathcal{C}_1$, see \eqref{m2onc1}. We have seen that $\mu_2(c)\to
1$ when $c\to 1/\sqrt{3}$, and that $\mu_2$ is decreasing for
$c\in(c_0,1/\sqrt{3})$ with $c_0=0.27448350\dots$, so $m_2-m_1>0$
near $P_1$. On the other hand at the point $Q=(2/\sqrt{3},
1/\sqrt{3},0)$ corresponding to the right trapezoid we have $ m_2=
0.15417163\dots <1$, see \eqref{massestr} which implies $m_2-m_1<0$.
Therefore there exists a point $T$ on $\mathcal{C}_1$ such that
$m_2=m_1$. Again applying Sturm Theorem to a polynomial of degree 22
in $c$, we can conclude that $T$ is unique, and its coordinates are
$(2/\sqrt{3},1/\sqrt{3},c_1)$ where $c_1= 0.0517595932\dots$

Let
$\widetilde{\mathcal{C}}_1=\{(a,b,c)\in\mathcal{C}_1\::\:c<c_1\}$
and
${\widetilde{\mathcal{C}}_1}'=\{(a,b,c)\in\mathcal{C}_1\::\:c>c_1\}$.
Therefore, for any given path connecting $\widetilde{\mathcal{C}}_1$
and $P_3\cup \mathcal{C}_3\cup P_1\cup {\widetilde{\mathcal{C}}_1}'$
in $\Omega$, there exist $(a,b,c)\in \Omega$ such that $m_2=m_1=1$.

Numerically we show that for any fixed path connecting
$\widetilde{\mathcal{C}}_1$ and $P_3\cup \mathcal{C}_3\cup P_1\cup
{\widetilde{\mathcal{C}}_1}'$ in $\Omega$ the solution of $m_2-1=0$
is unique. Curve representing the zeros of $m_2-1$ joins the
boundaries $\mathcal{C}_3$ and $\mathcal{C}_1$. This curve goes from
$(1.13102016\dots, 0.896392974\dots, 0.234627188\dots)\in
\mathcal{C}_3 $ to $(2/\sqrt{3},1/\sqrt{3},c_1)\in \mathcal{C}_1 $
(see Figure~\ref{fig1me1} (a)). The values of the masses $m_3$ and
$m_4$ along this curve are plotted in Figure~\ref{fig1me1} (b).

\begin{figure}
\includegraphics[width=5cm]{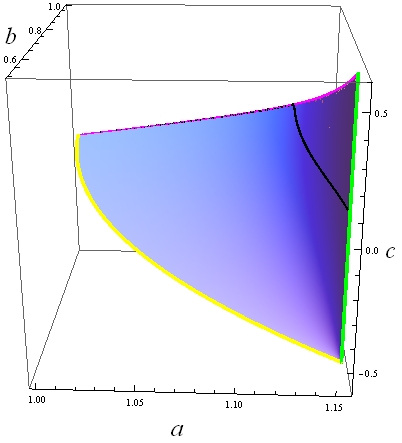} \hspace*{0.5cm}
\includegraphics[width=6.5cm]{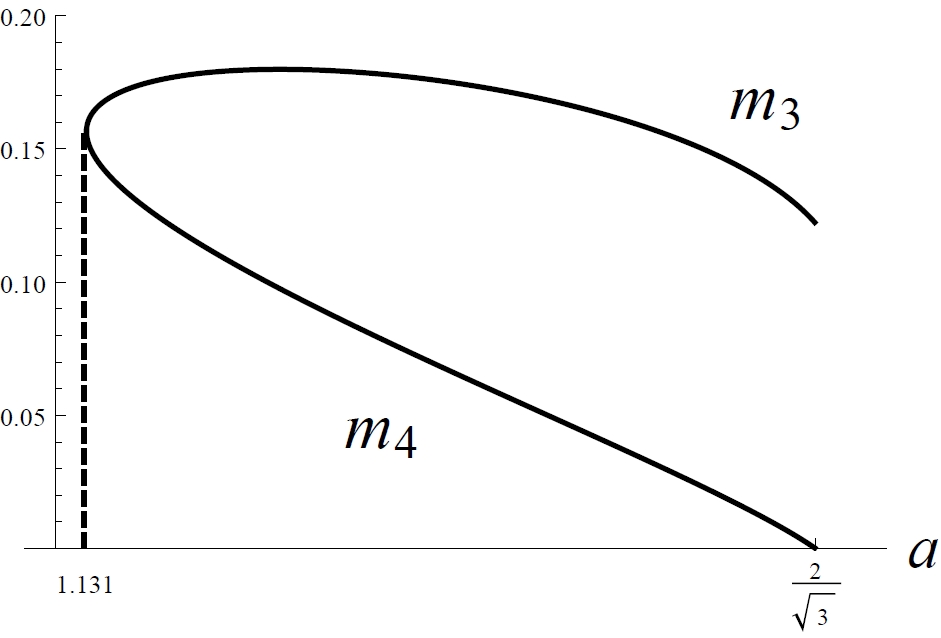}\\
(a) \hspace*{6cm} (b) \caption{(a) In black, the curve $m_2=1$ on
the region $\Omega$. (b) The curves of $m_3$ and $m_4$,
parameterized by $a$ on the curve in (a).} \label{fig1me1}
\end{figure}

\section{Conclusions}

Using the positions of the masses we have classified the set of
trapezoid central configurations. This set is a two--dimensional
surface whose boundaries are known families consisting in a rhombus,
an isosceles trapezoid and an equilateral triangle with a zero mass
off the triangle. Although a specific ordering of the masses has not
hold for any trapezoid central configuration, we can split the
two--dimensional surface in three disjoint regions where the set of
masses is totally ordered. Somewhat we must remark that we have
proved analytically the existence of non--symmetric trapezoid
central configurations with a pair of equal masses.

There exist a one--parameter family of right trapezoid central
configurations that also splits the two--dimensional surface in two
disjoint regions, namely the acute and the obtuse regions. Along
such a non--symmetric family the masses are completely ordered, that
is, the family belong to one of the previous three regions,
concretely the middle one, where the set of masses is totally
ordered. Moreover, when the pair of equal masses belong to biggest
parallel side, only acute trapezoid central configurations are
allow. On the other hand, when the two equal masses belongs to the
non--parallel side, only obtuse trapezoid central configurations are
allow.


\subsection*{Acknowledgements}

The first three authors are partially supported by a FEDER-MINECO
grant MTM2016-77278-P and a MINECO grant MTM2013-40998-P. The second
and third authors are also supported by an AGAUR grant number
2014SGR-568. The fourth author is supported by Fondo Mexicano de
Cultura A.C.

\end{document}